\documentclass[12pt,draftcls,onecolumn]{IEEEtran}

\usepackage{times,comment,color,graphicx,setspace,bbm,mathdots,mathrsfs,amssymb,latexsym,amsfonts,amsmath,cite,stmaryrd,caption,pgf,tikz}

\usepackage{enumerate,epstopdf,ifpdf,psfrag,epsfig}

\ifpdf
  \DeclareGraphicsExtensions{.pdf,.jpg,.png}
\else
  \DeclareGraphicsExtensions{.eps}
\fi
\ifCLASSINFOpdf

\newcommand{\V}{{\mathcal{V}}}
\newcommand{\s}{{\mathcal{S}}}
\newcommand{\e}{{\mathcal{E}}}
\newcommand{\h}{{\mathcal{H}}}

\DeclareMathOperator{\rank}{rank}
\DeclareMathOperator{\nullity}{nullity}

\DeclareMathOperator{\Rank}{{\mathbf{rank}}}
\DeclareMathOperator{\Nullity}{{\mathbf{nullity}}}
\DeclareMathOperator{\Null}{{\mathbf{null}}}

\begin{document}

\title{\'Eminence Grise Coalitions: On the Shaping of Public Opinion}

\author{Sadegh~Bolouki,~
        Roland~P.~Malham\'e,~
        Milad~Siami,~
        and~Nader~Motee
\thanks{S.~Bolouki, M.~Siami, and N.~Motee are with the Department
of Mechanics and Mechanical Engineering, Lehigh University, Bethlehem,
PA, 18015 USA e-mail: \{bolouki,siami,motee\}@lehigh.edu.}
\thanks{R.P.~Malham\'e is with the Department of Electrical Engineering, Polytechnique Montr\'eal, Montreal, QC, H3T 1J4 CA e-mail: roland.malhame@polymtl.ca.}
}

\maketitle


\begin{abstract}

	We consider a network of evolving opinions. It includes multiple individuals with first-order opinion dynamics defined in continuous time and evolving based on a general exogenously defined time-varying underlying graph. In such a network, for an arbitrary fixed initial time, a subset of individuals forms an \textit{\'eminence grise coalition}, abbreviated as EGC, if the individuals in that subset are capable of leading the entire network to agreeing on any desired opinion, through a cooperative choice of their own initial opinions. In this endeavor, the coalition members are assumed to have access to full profile of the underlying graph of the network as well as the initial opinions of all other individuals. While the complete coalition of individuals always qualifies as an EGC, we establish the existence of a minimum size EGC for an arbitrary time-varying network; also, we develop a non-trivial set of upper and lower bounds on that size. As a result, we show that, even when the underlying graph does not guarantee convergence to a global or multiple consensus, a generally restricted coalition of agents can steer public opinion towards a desired global consensus without affecting any of the predefined graph interactions, provided they can cooperatively adjust their own initial opinions. Geometric insights into the structure of EGC's are given. The results are also extended to the discrete time case where the relation with Decomposition-Separation Theorem is also made explicit.

\end{abstract}


\newtheorem{theorem}{Theorem}
\newtheorem{definition}{Definition}
\newtheorem{lemma}{Lemma}
\newtheorem{corollary}{Corollary}
\newtheorem{remark}{Remark}
\newtheorem{proposition}{Proposition}
\newtheorem{conjecture}{Conjecture}

%
\IEEEpeerreviewmaketitle


\section{Introduction}

	In this paper, we are mainly concerned with the occurrence of consensus in networks of individuals with opinions updated via a class of continuous time weighted distributed averaging algorithms characterized in general by an exogenous underlying chain of opinion update matrices, which behave like intensity matrices of a continuous time Markov chain. In such networks, \textit{consensus} is said to occur if all opinions converge to the same value as time grows large. Furthermore, \textit{Multiple consensus} is said to occur if each individual's opinion asymptotically converges to an individual limit. It is well known that such asymptotic behaviors relate directly to the properties of the Markov chain which underlies the opinion update dynamics. More specifically, the underlying chain of an opinion network may be such that consensus or multiple consensus occurs \textit{unconditionally}, i.e., irrespective of the values of initial opinions of the individuals in the network. The unconditional occurrence of consensus is proved to be equivalent to \textit{ergodicity} of the underlying chain \cite{Chatterjee:77}. There is a similar correspondence between the unconditional occurrence of multiple consensus and \textit{class-ergodicity} of the underlying chain \cite{Bol:13,Touri:3}.

	Ergodic and class-ergodic chains, i.e., chains leading to unconditional consensus or multiple consensus, have attracted an increasing attention in the literature in the past decade. Researchers of many different fields including robotics, social networks, economics, biology, etc., have been particularly interested in conditions under which a consensus algorithm guarantees consensus or multiple consensus to occur for an arbitrary choice of initial opinions. It is generally accepted that the earliest work on this class of opinion formation models was done in \cite{DeGroot:74}. The model was defined in discrete time, and the considered underlying chain was time-invariant. Later, more general cases were considered in \cite{Chatterjee:77}, where the authors also made explicit the relationship between consensus and ergodicity of the underlying chain. Some of the earliest significant results on consensus date back to \cite{Tsit:84,Tsit:86,Tsit:89a}. Interest in distributed consensus for agents moving in space was triggered by the numerical experiments in \cite{Vicsek:95} where a nonlinear algorithm was proposed for modeling evolution of multi-agent systems in discrete time. In this model, agents are assumed to have the same speed but different headings, and states are headings of agents. Using simulations, convergence to some kind of consensus (emerging behavior) was displayed in \cite{Vicsek:95}. A linearized version of the model in \cite{Vicsek:95} was considered in \cite{Jadba:03}, where sufficient conditions for consensus based on analyzing infinite products of stochastic matrices, consistent with those of \cite{Tsit:84,Tsit:86,Tsit:89a} are established. Following \cite{Jadba:03}, many works have focused on identifying the largest class of underlying update chains for which consensus occurs unconditionally. Because of their close relationship to our current work, we mention in particular \cite{Tsit:05,Moreau:05,Hend:06,Li:04,Lorenz:05,Hend:11,Touri:10a,Touri:10b,Touri:11c,Touri:3,Bolouki:ifac,Bolouki:11,Bolouki:12,Bolouki:12b}. In addition, \cite{Lorenz:05,Touri:10b,Touri:11c,Touri:3,Hend:11,Bolouki:ifac,Bolouki:11,Bolouki:12,Bolouki:12b,Bol:13} also addressed the unconditional multiple consensus problem, or equivalently class-ergodicity of the underlying chain. For the continuous time case, \cite{Hend:11} appears to provide the most general results thus far on consensus and multiple consensus. On the other hand, in our recent work \cite{Bol:13}, inspired by \cite{Touri:11c} and \cite{Sonin:08}, and to the best of our knowledge, we have identified for the discrete time case, the largest class to date of ergodic and class-ergodic chains.

	In contrast to the above papers, which are concerned with ``unconditional'' consensus, the current paper aims at providing some answers to the following questions: What if the underlying chain is not ergodic? How can consensus still be achieved in a network with absolutely no assumption on the underlying chain? In other words, for a network with a general time-varying underlying opinion update chain, having fixed the initial time, what can be said about particular (non-trivial) choices of initial opinions leading to a possible consensus? Geometric insights on the nature of the ``march'' towards consensus allow one to realize that such choices of initial opinion vectors form a vector space the dimension of which is related to the characteristics of the underlying chain. The fact that such initial opinion vectors form a vector space suggests the existence of a possibly small subgroup of individuals in the network who are \textit{naturally} capable of leading the whole group to eventually agree on any desired value \textit{only} by collectively adjusting their own initial opinions. The word ``naturally'' here refers to the fact that the subgroup does not need to manipulate the nature of the network, and particularly \textit{leaves all the interactions between any two individuals including themselves untouched}. They act like hidden leaders, or ``\'eminences grises'', not identifiable by title or position, yet who can, given time, thoroughly shape the ultimate public opinion. A subgroup with such leadership property is referred to as an \textit{\'Eminence Grise Coalition}, or simply EGC, in this work. The EGC's that a network admit are determined by the properties of the underlying chain of the network only. While it is trivial to establish the existence of at least one largest EGC, namely the universal coalition of individuals, one of our main points of interest in this work is to characterize the size and identity of the smallest coalition that can achieve public opinion shaping. Tight bounds on the size of that coalition are also of interest. The reasons why such individuals may want to act as a coalition can be multiple. Two such possibilities are: (i) They have been identified as key decision makers by a knowledgeable negotiator, have collectively agreed on a bargaining position, yet need to steer their peers towards the collective agreement, (ii) A shady opinion manipulator has identified them as key decision makers and has succeeded in ``buying out'' their collaboration.

	The rest of the paper is organized in such a way that no confusion arises between the continuous time and the discrete time cases. We explicitly deal with the continuous time case in the largest part of the paper, that is Sections \ref{problem setup}--\ref{full-rank chains}, and discuss the discrete time case in Section \ref{discrete time analysis}. More specifically, we explicitly state the problem setup in Section \ref{problem setup}, where we introduce the notion of \textit{rank} of a chain of matrices which is shown to be equal to the size of the smallest EGC of the network. In Section \ref{a geometric framework}, a geometric framework is developed to interpret the notion of rank of a chain and also obtain an upper bound for the $\text{rank}$, or equivalently the size of the smallest EGC of a consensus algorithm. This geometric framework proves to be useful in dealing with both the continuous time and the discrete time cases. We establish in Section \ref{lower bounds}, lower bounds on the $\text{rank}$ based on the existing notions in the literature, namely the so-called infinite flow graph and unbounded interactions graph of a chain. The $\text{rank}$ of time-invariant chains is discussed in Section \ref{rank of TI chains}. We address a large class of time-varying chains, the so-called Class $\mathcal{P}^*$, and their rank in particular, in Section \ref{Class P*}. It is shown that chains of the the two classes discussed in Sections \ref{rank of TI chains} and \ref{Class P*}, are examples of chains for which the bounds on $\text{rank}$ obtained earlier in Sections \ref{a geometric framework} and \ref{lower bounds} are actually attained. Full-rank chains, namely chains with $\text{rank}$ equal to the size of the network are characterized in Section \ref{full-rank chains}. In the process of characterizing full-rank chains, we also discover another upper bound on $\text{rank}$. In Section \ref{discrete time analysis}, we extend our analysis of the continuous time case to the discrete time case. As will be shown, the size of the smallest EGC is equal to the number of jets in the jet decomposition of the Sonin Decomposition Separation Theorem (see \cite{Sonin:08,Bol:13}). Concluding remarks and suggestions of future work end the paper in Section \ref{conclusion}.


\section{Notions and Terminology}
\label{problem setup}

 	The notions, preliminaries, and notation described in this section are for the purposes of the continuous time part of this paper, i.e., Sections \ref{problem setup}--\ref{full-rank chains}, although some may be consistent with the contents of Section \ref{discrete time analysis}, the discrete time analysis. Let $N$ be the number of individuals and $\V=\{1,\ldots,N\}$ be the set of individuals. Assume that $t$ stands for the continuous time index. Let a time-varying chain $\{A(t)\}_{t \geq 0}$ of square matrices of size $N$ be such that each matrix $A(t)$, $t \geq 0$, has zero row sum and non-negative off-diagonal entries and each entry $a_{ij}(t)$ of $A(t)$, $i,j \in \V$, is a measurable function. Continuous time chains of matrices, that we deal with in this paper, are assumed to have these properties. According to these constraints, $A(t)$ can be viewed as the evolution of the intensity matrix of a time inhomogeneous Markov chain. Let dynamics of an opinion network be described by the following continuous time distributed averaging algorithm:
	\begin{equation}
		\dot{x}(t) = A(t)x(t), \, t \geq t_0,
	\label{mc}
	\end{equation}
where $t_0 \geq 0$ is the initial time and $x(t) \in \mathbb{R}^N$ is the vector of opinions at each time instant $t \geq t_0$. Thus, $x_i(t)$ is the scalar opinion of individual $i$ at time $t \geq t_0$. Chain $\{A(t)\}_{t \geq 0}$, or simply $\{A(t)\}$, is referred to as the \textit{underlying chain} of the network with dynamics (\ref{mc}).

	Assume that $\Phi(t,\tau)$, $t \geq \tau \geq 0$ denotes the state transition matrix associated with chain $\{A(t)\}$. Therefore, for the network with dynamics (\ref{mc}), we must have:
	\begin{equation}
		x(t) = \Phi(t,\tau) x(\tau), \, \forall t \geq \tau \geq t_0.
	\label{mgeneral}
	\end{equation}
	From \cite[Section 1.3]{Brockett:70}, the Peano-Baker series of state transition matrix $\Phi(t,\tau)$, $t \geq \tau \geq 0$, associated with chain $\{A(t)\}$ is expressed as:
	\begin{equation}
		\begin{array}{ll}
			\Phi(t,\tau) = I_{N\times N}	& \hspace{-.1in} + \int_{\tau}^t A(\sigma_1)d\sigma_1 \vspace{.05in}\\
											& \hspace{-.1in} + \int_{\tau}^t A(\sigma_1) \int_{\tau}^{\sigma_1} A(\sigma_2) d\sigma_2 d\sigma_1 \vspace{.05in}\\
					    						& \hspace{-.1in} + \int_{\tau}^t A(\sigma_1) \int_{\tau}^{\sigma_1} A(\sigma_2) \int_{\tau}^{\sigma_2} A(\sigma_3) d\sigma_3 d\sigma_2 d\sigma_1 \vspace{.05in}\\
					    						& \hspace{-.1in}+ \cdots,
  		\end{array}
  	\label{state transition matrix}
	\end{equation}
where $I_{N \times N}$ denotes the identity matrix of size $N$. Remember that state transition matrix $\Phi(t,\tau)$ is invertible for every $t \geq \tau \geq 0$.

	We use the following notation throughout this paper: $\Phi_{i}(t,\tau)$ and $\Phi_{i,j}(t,\tau)$, $1 \leq i,j \leq N$, denote the $i$th column and the $(i,j)$th element of $\Phi(t,\tau)$ respectively. Moreover, the transposition of a matrix is indicated by the matrix followed by prime ($'$). We emphasize that $\Phi'_i(t,\tau)$ refers to the $i$th column of $\Phi'(t,\tau)$ (prime acts first). For an arbitrary vector $v \in \mathbb{R}^N$, and $1 \leq i \leq N$, $v_i$ denotes the $i$th element of $v$. Vectors of all zeros and all ones in $\mathbb{R}^N$ are indicated by $\mathbf{0}_N$ and $\mathbf{1}_N$ respectively. For an arbitrary subset $\s \subset \V$, $\V \backslash \s$ denotes the complement of $\s$ in $\V$.
	
	\begin{remark}
		Notice that $\Phi_{i,j}(t,\tau)$, $t \geq \tau \geq 0$, for a fixed $\tau$, can be viewed as a transition probability in a backward propagating inhomogeneous Markov chain. In particular, for every $t_2 \geq t_1 \geq \tau \geq 0$, we have:
  	\begin{equation}
    		\Phi_{i,j}(t_2,\tau) = \sum_{k} \Phi_{i,k}(t_2,t_1) \Phi_{k,j}(t_1,\tau),
  	\end{equation}
with the conditions:
  	\begin{equation}
    		\Phi_{i,j}(t,\tau) \geq 0,
  	\end{equation}
  	\begin{equation}
    		\sum_{j} \Phi_{i,j}(t,\tau) = 1,
  	\end{equation}
  	\begin{equation}
    		\Phi_{i,j}(\tau,\tau) = \delta_{ij},
  	\end{equation}
where $\delta_{ij}$ is the Kronecker symbol.
	\end{remark}


\subsection{\'Eminence Grise Coalition}

	\begin{definition}
	\label{EGC def}
	For an opinion network with dynamics (\ref{mc}), a subgroup of individuals $\s \subset \V$ is said to be an \textit{\'Eminence Grise Coalition} if for any arbitrary $x^* \in \mathbb{R}$ and any initialization of opinions of individuals in $\V \backslash \s$, there exists an initialization of opinions of individuals in $\s$ such that $\lim_{t \rightarrow \infty} x(t) = x^*.\textbf{1}_N$, i.e., all individuals asymptotically agree on $x^*$. The term \'Eminence Grise Coalition may also be referred to as acronym \textit{EGC}.
	\end{definition}
	
	From another point of view that also justifies the selection of the term \'Eminence Grise Coalition, an EGC of a network with dynamics (\ref{mc}) is a subgroup of individuals who are capable of leading the whole group towards a global agreement on any desired ultimate opinion only by properly initializing their own opinions, with the assumption that they are aware of the underlying chain of the network and initial opinions of the rest of individuals.
	
	\begin{lemma}
	\label{zero is enough}
		In an opinion network with dynamics (\ref{mc}), a subset $\s \subset \V$ is an EGC if and only if for any initialization of opinions of individuals in $\V \backslash \s$, there exists an initialization of opinions of individuals in $\s$ such that $\lim_{t \rightarrow \infty} x(t)= \mathbf{0}_N$.
	\end{lemma}
	
	\begin{proof}
		The ``only if'' part is obvious by setting $x^*=0$ in Definition \ref{EGC def}. Conversely, assume that $\s \subset \V$ has the property that for any initialization of individuals in $\V \backslash \s$, there exists an initialization of individuals in $\s$ such that all opinions asymptotically converge to zero. To show that $\s$ is an EGC according to Definition \ref{EGC def}, let arbitrary $x^* \in \mathbb{R}$ be the desired value of agreement and assume that for every $i \in \V \backslash \s$, the opinion of individual $i$ is initialized at $\hat{x}_i \in \mathbb{R}$, where $\hat{x}_i$ is arbitrary. We seek an initialization of opinions of individuals in $\s$ leading to an asymptotic agreement of all individuals on $x^*$. For a moment, let us assume that for every $i \in \V \backslash \s$, the opinion of individual $i$ was initialized at $\hat{x}_i - x^*$. For such an initialization, by the assumption on $\s$, there would be an initialization of opinions of individuals in $\s$, say at $\hat{x}_i$ for each individual $i \in \s$, such that all opinions would asymptotically converge to zero. In other words, if  the individual opinions in the network with dynamics (\ref{mc}) were initialized as:
		\begin{equation}
			x_i(0) =
			\begin{cases}
				\hat{x}_i - x^*	& \text{ if } i \in \V \backslash \s\\
				\hat{x}_i		& \text{ if } i \in \s
			\end{cases}
		\end{equation}
then, $\lim_{t \rightarrow \infty} x(t)= \mathbf{0}_N$. Now, the following initialization, which is basically a translation of the previous initialization by $x^*$, will lead to an agreement on $x^*$:
		\begin{equation}
			x_i(0) =
			\begin{cases}
				\hat{x}_i			& \text{ if } i \in \V \backslash \s\\
				\hat{x}_i+x^*		& \text{ if } i \in \s
			\end{cases}
		\end{equation}
		Agreement on $x^*$ is easily proved from the previous agreement on zero and noticing that translations are preserved in consensus dynamics (\ref{mc}) since $\Phi(t,t_0)$, for every $t \geq t_0$, has an eigenvector $\mathbf{1}_N$ corresponding to eigenvalue 1. Thus, for an arbitrary initialization of individuals in $\V \backslash \s$, we found an initialization of individuals in $\s$ such that all opinions asymptotically converge to the desired value $x^*$, which completes the proof.
	\end{proof}
	
	Our primary objective in this work is characterizing the smallest EGC in an opinion network with dynamics described by (\ref{mc}). In particular, the size of the smallest EGC is of interest.


\subsection{Rank of a Chain}
	
	We now define several operators for chains of matrices. $\mathbf{Bold}$ style is used for chain operators in this paper to distinguish them from matrix operators that are in $\mathrm{roman}$ style. Let $\{A(t)\}$ be a chain of matrices and $\Phi(t,\tau)$, $t \geq \tau \geq 0$ be its associated state transition matrix.
	\begin{definition}
	\label{null space}
  		The \textit{null space} of chain $\{A(t)\}$ at time $\tau \geq 0$, denoted by $\Null_{\tau}(A)$, is defined by:
  		\begin{equation}
    			\Null_{\tau}(A) \triangleq \Big\{ v \in \mathbb{R}^N | \lim_{t \rightarrow \infty} \big( \Phi(t,\tau)v \big) = \textbf{0}_N  \Big\}.
    		\label{null-space}
  		\end{equation}
	\end{definition}
	It is straightforward to show that $\Null_{\tau}(A)$ is a vector space for every $\tau \geq 0$.
	
	\begin{lemma}
		The dimension of vector space $\Null_{\tau}(A)$, $\tau \geq 0$, is independent of $\tau$.
	\end{lemma}
	
	\begin{proof}
		Let $\tau_2 > \tau_1 \geq 0$ be two arbitrary time instants. Define linear operator $\phi_{\tau_2,\tau_1}:\mathbb{R}^N \rightarrow \mathbb{R}^N$ by:
  		\begin{equation}
    			\phi_{\tau_2,\tau_1}(v) \triangleq \Phi(\tau_2,\tau_1)v, \, \forall v \in \mathbb{R}^N.
    		\label{operator1}
  		\end{equation}
		Noticing that $\Phi(\tau_2,\tau_1)$ is invertible, it is not difficult to see that operator $\phi_{\tau_2,\tau_1}$ creates a one-to-one correspondence between the two vector spaces $\Null_{\tau_1}(A)$ and $\Null_{\tau_2}(A)$. As a result, the two vector spaces are of equal dimensions.
	\end{proof}
	
	\begin{definition}
		The constant dimension of $\Null_{\tau}(A)$, $\tau \geq 0$, which is independent of $\tau$, is called \textit{nullity} of chain $\{A(t)\}$ and is denoted by $\Nullity(A)$. Moreover, the \textit{rank} of chain $\{ A(t) \}$ is defined by:
		\begin{equation}
  			\Rank(A) \triangleq N - \Nullity(A).
		\label{rank-nullity}
		\end{equation}
	\end{definition}
	
	The following theorem suggests that one can investigate the size of the smallest EGC via the notion of $\text{rank}$.

	\begin{theorem}
  		For an opinion network with dynamics described by (\ref{mc}), the size of the smallest EGC is $\Rank(A)$.
  	\label{rank=EGC}
	\end{theorem}

	\begin{proof}
  		To simplify the proof, let $r \triangleq \Rank(A)$ and $h$ be the size of the smallest EGC. Our aim is to show that $r=h$. Equivalently, we prove, in the following, that $h \leq r$ and $r \leq h$.
		
  		$(h \leq r)$: We show that there is an EGC of size $r$. From Lemma \ref{zero is enough}, it suffices to show that there exists a subset $\s \subset \V$ of size $r$ with the property that for any initialization of the opinions of individuals in $\V\backslash \s$, there exists an initialization of the opinions of individuals in $\s$ such that all opinions asymptotically converge to zero. Note that $\Null_{t_0}(A)$ is a vector space with dimension $\Nullity(A)=N-r$. Let $\beta_1,\ldots,\beta_{N-r}$ be a basis of $\Null_{t_0}(A)$. Notice that the column-rank of matrix
		\begin{equation}
			\begin{bmatrix}
				\beta_1 | \cdots | \beta_{N-r}
			\end{bmatrix}
		\label{basis}
		\end{equation}
		is $N-r$, and so is its row-rank. Thus, matrix (\ref{basis}) has $N-r$ linearly independent rows. Note that the choice of the $N-r$ linearly independent rows is not necessarily unique. Assume that $i_1,\ldots,i_{N-r}$ are the indices of $N-r$ independent rows of matrix (\ref{basis}). It is straightforward to show that subset $\s \subset \V$ defined by:
		\begin{equation}
			\s = \V \backslash \{i_1,\ldots,i_{N-r}\},
		\end{equation}
		has the desired property.
		
		$(r \leq h)$: Since there exists an EGC of size $h$, there are $N-h$ individuals such that no matter what their initial opinions are, there is an initial opinion vector that results in all opinions asymptotically going to zero, or equivalently, an initial opinion vector that belongs to $\Null_{t_0}(A)$. Thus, vector space $\Null_{t_0}(A)$ has dimension greater than or equal to $N-h$, i.e., $N-r \geq N-h$.
	\end{proof}

	\begin{remark}
  		Another point of interest regarding the issue of consensus, that we will not further discuss in this work, is that of the nature of the set of initial opinion vectors leading to consensus in the network with dynamics (\ref{mc}); more precisely:
  		\begin{equation}
    			\{ x(t_0) |\, \exists x^* \in \mathbb{R}: \lim_{t \rightarrow \infty}x(t) = x^*. \textbf{1}_N \},
    		\label{consensus-set}
  		\end{equation}
  		It is straightforward to see that set (\ref{consensus-set}) is the vector space generated by $\Null_{t_0}(A)$ and $\textbf{1}_N$. Consequently, vector space (\ref{consensus-set}) has dimension $\Nullity(A)+1$.
	\end{remark}

	Keeping Theorem \ref{rank=EGC} in mind, we focus on the notion of $\text{rank}$ in the rest of the paper. In the following, we give the continuous time version of the definition of $l_1$-approximation initially introduced in \cite{Touri:10b} for discrete time chains.

	\begin{definition}
  		Chain $\{A(t)\}$ is said to be an \textit{$l_1$-approximation} of chain $\{B(t)\}$ if:
  		\begin{equation}
    			\int_{0}^{\infty} \| A(t)-B(t) \| dt < \infty,
  		\end{equation}
  	where for convenience only, the norm refers to the \textit{max norm}, i.e., the maximum of the absolute values of the matrix elements.
	\end{definition}

	It is not difficult to show that $l_1$-approximation is an equivalence relation in the set of chains that are candidates of the underlying chain of an opinion network. The importance of the $l_1$-approximation notion in this work comes from the following lemma. The proof is eliminated due to its similarity to the proof of \cite[Lemma 1]{Touri:10b}.

	\begin{lemma}
  		The $\text{rank}$ of a chain is invariant under an $l_1$-approximation.
  	\label{rank and approximation}
	\end{lemma}


\subsection{Ergodicity and Class-Ergodicity}
	
	Several other definition related to chains of matrices will be needed and are given as follows.
	
	\begin{definition}
		Chain $\{A(t)\}$ is said to be \textit{ergodic} if for every $\tau \geq 0$, its associated state transition matrix $\Phi(t,\tau)$ converges to a matrix with equal rows as $t \rightarrow \infty$.
	\end{definition}
		
	From \cite{Chatterjee:77}, we know that ergodicity of $\{A(t)\}$ is equivalent to the occurrence of unconditional consensus in (\ref{mc}).
	
	\begin{definition}
		Chain $\{A(t)\}$ is \textit{class-ergodic} if for every $\tau \geq 0$, $\lim_{t \rightarrow \infty}\Phi(t,\tau)$ exists but has possibly distinct rows.
	\end{definition}
	
	It is known that chain $\{ A(t) \}$ is class-ergodic if and only if multiple consensus occurs in (\ref{mc}) unconditionally (see \cite{Bol:13,Touri:3}). We define, in what follows, the ergodicity classes of a chain according to \cite{Touri:10b}.
	\begin{definition}
	\label{ergodicity classes}
		For an opinion network with state transition matrix $\Phi(t,\tau)$, $t \geq \tau \geq 0$, two individuals $i,j \in \V$ are said to be \textit{mutually weakly ergodic} if and only if for every $\tau \geq 0$:
  		\begin{equation}
    			\lim_{t \rightarrow \infty} \| \Phi'_i (t,\tau) - \Phi'_j(t,\tau) \| = 0.
		\label{w-m-e}
  		\end{equation}
	\end{definition}
	
	It is easy to see that the relation of being mutually weakly ergodic is an equivalence relation on $\V$. The equivalence classes of this relation are referred to as \textit{ergodicity classes} in this paper. Indeed, these equivalence classes form a partitioning of $\V$, and while in some cases they may simply be singletons, they can always be defined  for an arbitrary chain \{A(t)\}. If chain $\{A(t)\}$ is class-ergodic, i.e,. $\lim_{t \rightarrow \infty} \Phi'_i (t,\tau)$ exists for every $i \in \V$ and $\tau \geq 0$, then $i,j \in \V$ are in the same ergodicity class if $\lim_{t \rightarrow \infty} \Phi'_i (t,\tau) = \lim_{t \rightarrow \infty} \Phi'_i (t,\tau)$, for every $\tau \geq 0$. We refer to the ergodicity classes of a class-ergodic chain as \textit{ergodic classes}.


\section{A Geometric Interpretation of the $\text{Rank}$}
\label{a geometric framework}

	In this Section, we employ a geometric approach to analyze the asymptotic properties of a chain of matrices . This approach, which can be used for both the continuous and discrete time cases, will help us to (i) geometrically interpret the $\text{rank}$ of a general time-varying chain, (ii) identify an upper bound for the $\text{rank}$, and (iii) investigate the limiting behavior of a large class of time-varying chains, namely Class $\mathcal{P}^*$ as discussed in Section \ref{Class P*}.

	For time-varying chain $\{A(t)\}_{t \geq 0}$, define $C_{t,\tau}$, $t \geq \tau \geq 0$ as the convex hull of points in $\mathbb{R}^N$ corresponding to the columns of the transpose of associated state transition matrix $\Phi(t,\tau)$. Note that $C_{t,\tau}$ is a polytope, with no more than $N$ vertices, in $\mathbb{R}^N$. We recall that each column of $\Phi'(t,\tau)$ is a stochastic vector, i.e., its elements are non-negative and add up to 1. We now have the following lemma regarding convex hull $C_{t,\tau}$.

	\begin{lemma}
  		For every $t_2 \geq t_1 \geq \tau$, we have: $C_{t_2,\tau} \subset C_{t_1,\tau}$, i.e., polytopes $C_{t,\tau}$, for an arbitrary fixed $\tau$, form a monotone decreasing sequence of polytopes in $\mathbb{R}^N$.
	\label{convex-hull}
	\end{lemma}

	\begin{proof}
		Note that:
  		\begin{equation}
    			\Phi(t_2,\tau) = \Phi(t_2,t_1)\Phi(t_1,\tau),
  		\label{nested-1}
  		\end{equation}
  		or equivalently,
  		\begin{equation}
    			\Phi'(t_2,\tau) =  \Phi'(t_1,\tau) \Phi'(t_2,t_1)
  		\label{nested}
  		\end{equation}
		Since $\Phi'(t_2,t_1)$ is a column-stochastic matrix, relation (\ref{nested}) implies that each column of $\Phi'(t_2,\tau)$ is a convex combination of the columns of $\Phi'(t_1,\tau)$. Therefore, each column of $\Phi'(t_2,\tau)$ lies in or on $C_{t_1,\tau}$, and the lemma is proved.
	\end{proof}

	Lemma \ref{convex-hull} shows that for a fixed $\tau \geq 0$, polytopes $C_{t,\tau}$'s, $t \geq \tau$, are nested in $\mathbb{R}^N$. An example of these nested polytopes projected on a two-dimensional subspace of $\mathbb{R}^N$ is depicted in Fig. \ref{Fig1}.

	\begin{figure}[h]
  		\begin{center}
  			\captionsetup{justification=centering}
    			\includegraphics[scale=1]{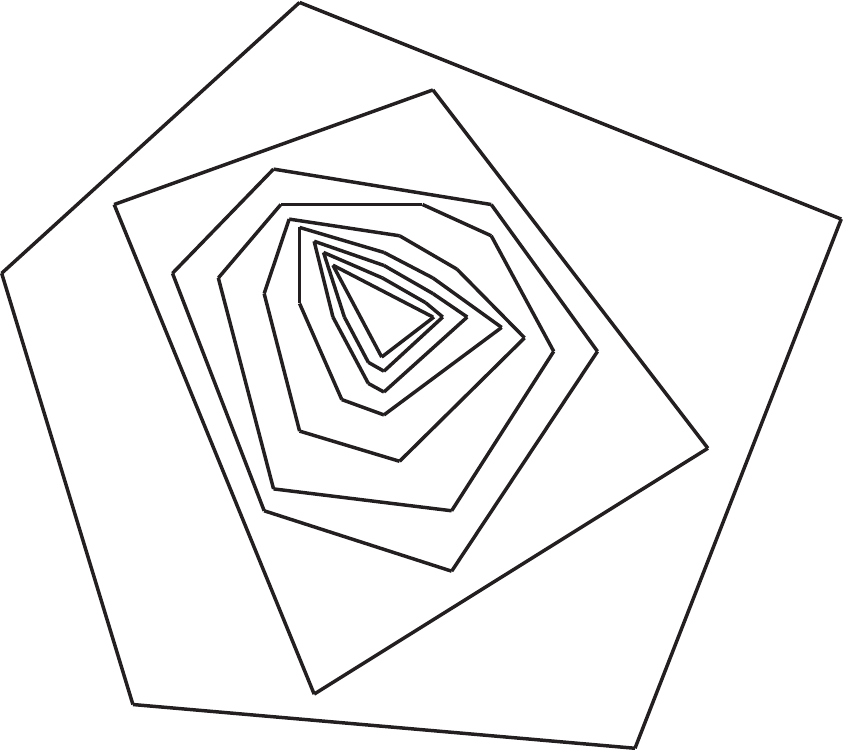}
    			\caption{Nested polygons converging to a triangle.}
  	\label{Fig1}
		\end{center}
	\end{figure}
	
	Note that for every $\tau \geq 0$, $\lim_{t \rightarrow \infty} C_{t,\tau}$ exists and is also a polytope in $\mathbb{R}^N$ due to the existence of a uniform upper bound, namely $N$, on the number of vertices of the nested polytopes. Let $C_{\tau}$ denote the limiting polytope and $c_{\tau}$ be the number of its vertices.

	\begin{lemma}
  		$c_{\tau}$, $\tau \geq 0$, is independent of $\tau$.
  	\label{fixed vertices}
	\end{lemma}

	\begin{proof}
  		Assume that $\tau_2 \geq \tau_1 \geq 0$ are two arbitrary time instants. Define linear operator $\phi'_{\tau_2,\tau_1}:\mathbb{R}^N \rightarrow \mathbb{R}^N$ by:
  		\begin{equation}
    			\phi'_{\tau_2,\tau_1}(v) \triangleq \Phi'(\tau_2,\tau_1)v, \, \forall v \in \mathbb{R}^N.
    		\label{operator}
  		\end{equation}
		Note now that from (\ref{nested}), for $t \geq \tau_2 \geq \tau_1 \geq 0$ we have:
  		\begin{equation}
    			\Phi'(t,\tau_1) =  \Phi'(\tau_2,\tau_1) \Phi'(t,\tau_2).
  		\label{nested-2}
  		\end{equation}
		Therefore, in view of (\ref{nested-2}) by taking $t$ to infinity, the vertices of $C_{\tau_2}$ are uniquely mapped to vectors in $\mathbb{R}^N$ which because of the linearity of map (\ref{operator}), will play the role of vertices for the generation of convex hull $C_{\tau_1}$. Also, it is not difficult to show that the images of vertices of $C_{\tau_2}$ must remain vertices of $C_{\tau_1}$, for if one of the images of a vertex of $C_{\tau_2}$, say $v$,  turned out to be a convex combination of other vertices of $C_{\tau_1}$, this would also be true for the inverse images of these vertices (also vertices of $C_{\tau_2}$ due to invertibility of matrix $\Phi'(\tau_2,\tau_1)$), and $v$ would then fail to be a vertex of $C_{\tau_2}$. In conclusion, $C_{\tau_1}$ and $C_{\tau_2}$ will have the same number of vertices, and (\ref{operator}) constitutes a one to one map between corresponding pairs of vertices.
	\end{proof}
	
	 Let integer $c$ be the constant value of $c_{\tau}$, $\tau \geq 0$. We will show later in this section that $c$ is equal to $\Rank(A)$. To prove this, we first state the following two lemmas.

	\begin{lemma}
  	\label{space-generated}
  		$\Rank(A)$ is equal to the dimension of the vector space generated by the vectors corresponding to the vertices of $C_{\tau}$, for every $\tau \geq 0$.
	\end{lemma}

	\begin{proof}
  		It suffices to prove Lemma \ref{space-generated} for $\tau = 0$. Let $v_1,\ldots,v_{c} \in \mathbb{R}^N$ be the $c$ vertices of $C_0$. It is easy to see that for any $u \in \mathbb{R}^N$:
  		\begin{equation}
    			u \in \mathcal{N}_0(A)\, \Longleftrightarrow\, v'_i u = 0,\, \forall i, 1 \leq i \leq c.
  		\end{equation}
  		It implies that the dimension of the vector space generated by $v_1,\ldots,v_c$ is $N - \Nullity(A)$, which proves the lemma.
	\end{proof}

	\begin{lemma}
  		For every $\tau \geq 0$, the vectors corresponding to the vertices of $C_{\tau}$ are linearly independent.
  	\label{independence}
	\end{lemma}
	\begin{proof}
  		It is  sufficient to prove the lemma for $\tau = 0$, i.e., to show that the vertices of $C_0$, namely $v_1,\ldots,v_c$, are linearly independent. Assume that $\alpha_1,\ldots, \alpha_c \in \mathbb{R}$ are such that:
  		\begin{equation}
    		\label{dependent}
    			\sum_{i=1}^c \alpha_i v_i = 0.
  		\end{equation}
		We note that vector $v_i$, $1 \leq i \leq c$, must lie outside of the convex hull of vectors $v_j$'s, $j \neq i$, for otherwise it would not qualify as a vertex. For every $i$, $1 \leq i \leq c$, let $w_i$ be the projection of $v_i$ on the convex hull of $v_j$'s, $j \neq i$. Define the following positive numbers:
  		\begin{equation}
    			\epsilon \triangleq \frac{1}{4}\min \{ \|v_i-w_i\| \, | \, 1 \leq i \leq c \},
  		\end{equation}
  		and:
  		\begin{equation}
    			\epsilon_1 \triangleq \epsilon /(2N).
			\label{zzzzz}
  		\end{equation}
  		Because $C_0$ is the limit of $C_{t,0}$ as $t$ goes to infinity, there must exist a sufficiently large time $T \geq 0$, such that for $t \geq T$, every point in $C_{t,0}$ lies within an $\epsilon_1$-distance of $C_{0}$. As depicted in Fig. \ref{Fig2}, for every $i$, $1 \leq i \leq c$, let $l_i$ be the hyperplane in $\mathbb{R}^N$ distant $\epsilon$ from $v_i$, crossing segment $v_iw_i$ and orthogonal to it. Let also $m_i$ be the hyperplane which is parallel to $l_i$, on the other side of $v_i$, distant $\epsilon_1$ from $v_i$.
  
		\begin{figure}[h]
  			\begin{center}
  				\captionsetup{justification=centering}
    				\vspace{-.7in}
    				\includegraphics[scale=.4]{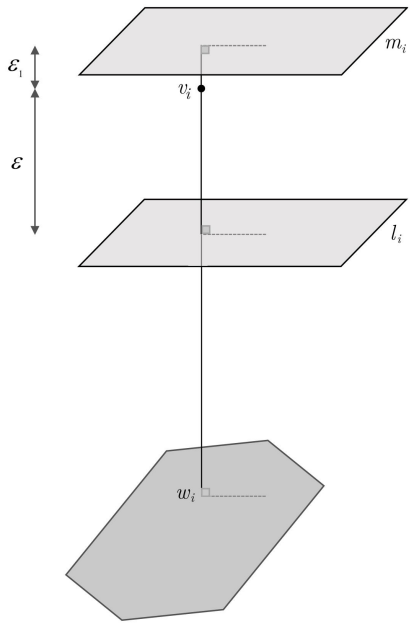}
    				\vspace{-.4in}
    				\caption{Planes $l_i$ and $m_i$ are orthogonal to segment $v_iw_i$.}
		\label{Fig2}
			\end{center}
		\end{figure}
  
  		Define for every $i$, $1 \leq i \leq c$:
  		\begin{equation}
    			S^i = \{ j \in \V \, | \, \Phi'_j(T,0) \text{ lies in the strip margined by } l_i,m_i \}.
  		\end{equation}  
  		Note that by the assumption, every point in $C_{T,0}$, including $\Phi'_j(T,0)$, lies within an $\epsilon_1$-distance of $C_{0}$. Therefore, $\Phi'_j(T,0)$ must lie on the same side of $m_i$ as $v_i$ does. In other words, $\Phi'_j(T,0)$ either lies in the strip margined by $l_i$ and $m_i$ or lies on the side of $l_i$ opposite to $v_i$ (below $l_i$ in Fig. \ref{Fig2}). This implies that $S_i$, $1 \leq i \leq c$, is non-empty. Indeed otherwise, $\Phi'_j(T,0)$ would lie below $l_i$ in Fig. \ref{Fig2} for every $j$ resulting in $C_{T,0}$ also lying below $l_i$, which would be a contradiction since $C_{T,0}$ must contain $C_0$ and $v_i$ in particular. One can also show that $S^i$'s, $1 \leq i \leq c$, are pairwise disjoint sets. More specifically, one can show that any point of $C_{T,0}$ that lies in the intersection of any two of sets $S^i$'s cannot be within $\epsilon$-distance of $C_0$, and since $\epsilon > \epsilon_1$, this would violate the defining property of $T$. $C_0$ being the limit of shrinking convex hulls $C_{t,0}$'s, it follows that for $i=1,\ldots,c$, there exists sequences $\{i_t\}$ of individuals such that $\Phi'_{i_t}(t,0)$ converges to $v_i$. Therefore, after some finite time, we have the following inequality:
  		\begin{equation}
    			\|\Phi'_{i_t}(t,0)-v_i\|<\epsilon_1.
    		\label{ineqq1}
  		\end{equation}
  		Without loss of generality, we can assume that the inequality (\ref{ineqq1}) holds for every $t \geq T$ (otherwise, we would proceed by replacing $T$ with $T'$, $T' > T$, such that inequality (\ref{ineqq1}) holds for every $t \geq T'$). We have for every $t \geq T$:
  		\begin{equation}
    			\begin{array}{ll}
      			\Phi'_{i_t}(t,0)	& \hspace{-.1in} = \Phi'(T,0) \Phi'_{i_t}(t,T) \vspace{.05in}\\
								& \hspace{-.1in} = \sum_{j \in \V}\Phi_{i_t ,j}(t,T) \Phi'_j(T,0) \vspace{.05in}\\
                             				& \hspace{-.1in} = \sum_{j \not\in S^i}\Phi_{i_t ,j}(t,T) \Phi'_j(T,0) + \sum_{j \in S^i}\Phi_{i_t ,j}(t,T) \Phi'_j(T,0).
    			\end{array}
    		\label{eqq1}
  		\end{equation}
  		We now show that for every $i$, $1 \leq i \leq c$, the following two inequalities must hold:
  		\begin{equation}
    			\sum_{j\not\in S^i} \Phi_{i_t, j}(t,T) < 2/(2N+1),
    		\label{tendd}
  		\end{equation}
  		\begin{equation}
    			\sum_{j \in S^i} \Phi_{i_t, j}(t,T) > 1- 2/(2N+1).
   	 	\label{tend}
  		\end{equation}
  		To prove (\ref{tendd}) and (\ref{tend}), we use (\ref{eqq1}) to find a lower bound for the distance from $\Phi'_{i_t}(t,0)$, $t \geq T$, to hyperplane $m_i$ as drawn in Fig. \ref{Fig2}. Remember that if $j \in S^i$, then, $\Phi'_j(T,0)$ lies in the strip margined by $m_i$ and $l_i$, while if $j \not\in S^i$, then, $\Phi'_j(T,0)$ lies below $l_i$ in Fig. \ref{Fig2}. For a fixed $i$, $1 \leq i \leq c$, let $\eta \triangleq \sum_{j\not\in S^i} \Phi_{i_t, j}(t,T)$. $\Phi(t,T)$ being row-stochastic, it immediately follows that $\sum_{j \in S^i} \Phi_{i_t, j}(t,T) = 1 - \eta$. Using (\ref{eqq1}), we now conclude that:
		\begin{equation}
    			\eta (\epsilon_1 + \epsilon) + (1-\eta).0
  		\end{equation}
is a lower bound for the distance from $\Phi'_{i_t}(t,0)$, $t \geq T$, to hyperplane $m_i$. This distance, on the other hand, is upper bounded by $2\epsilon_1$ since inequality (\ref{ineqq1}) is satisfied for every $t \geq T$. Thus, we must have:
		\begin{equation}
    			\eta (\epsilon_1 + \epsilon) + (1-\eta).0 < 2\epsilon_1,
  		\end{equation}
which immediately results in $\eta < 2/(2N+1)$ (remember that $\epsilon = 2N\epsilon_1$), and inequalities (\ref{tendd}) and (\ref{tend}) follow. Now remember by construction that $\lim_{t \rightarrow \infty} \Phi'_{i_t}(t,0) = v_i$ where $v_i$ is a given vertex of $C_0$. Furthermore, noting that:
  		\begin{equation}
    			\Phi'_{i_t}(t,0) = \Phi'(T,0) \Phi'_{i_t}(t,T),
  		\end{equation}
and taking limits on both sides as t goes to infinity, it follows that $\lim_{t \rightarrow \infty} \Phi'_{i_t}(t,T)$ is the image of a vertex of $C_0$ and therefore (following the proof of Lemma \ref{fixed vertices}) is itself a vertex of $C_T$, say $u_i$. Considering (\ref{tend}) again, and taking limits as $t \rightarrow \infty$, one can conclude:
  		\begin{equation}
    			\sum_{j \in S^i} (u_i)_j \geq 1- 2/(2N+1),
    		\label{w}
  		\end{equation}
and consequently:
  		\begin{equation}
    			\sum_{j \not\in S^i} (u_i)_j \leq 2/(2N+1).
    		\label{w1}
  		\end{equation}
  		Inequality (\ref{w}) can be established for $i = 1,\ldots,c$, where $u_i$, $i = 1,\ldots,c$ are the vertices of $C_T$.
  		Recalling linear operator $\phi_{\tau_2,\tau_1}$ from (\ref{operator}) one can write for some permutation $\sigma$ over set $\{1,\ldots,c\}$:
  		\begin{equation}
  		\label{u-v}
    			u_i = \Phi'(T,0) v_{\sigma(i)}, \, \forall i, \, 1\leq i \leq c,
  		\end{equation}
  		Combining relations (\ref{dependent}) and (\ref{u-v}) yields:
  		\begin{equation}
    			\sum_{i=1}^c \alpha_{\sigma(i)} u_i = 0,
    		\label{u-dependent}
  		\end{equation}
  		If we now assume that $k$, $1\leq k \leq c$, is such that:
  		\begin{equation}
    			| \alpha_{\sigma(k)} | = \max_{1 \leq i \leq c} \{ |\alpha_i| \} \triangleq \alpha,
  		\end{equation}
  		Now noting that (\ref{w}) and (\ref{w1}) hold only for the vertex $u_i$ which is the image of $v_i$, and that the $S^i$'s  are disjoint sets of agents, one can write the following:
  		\begin{equation}
    			\begin{array}{ll}
    				0 	& \hspace{-.1in} = |\sum_{j\in S^k}\sum_{i=1}^c \alpha_{\sigma(i)} (u_i)_j | \vspace{.05in}\\   
       				& \hspace{-.1in} = | \sum_{j\in S^k} \alpha_{\sigma(k)} (u_k)_j  + \sum_{j\in S^k}\sum_{i \neq k} \alpha_{\sigma(i)} (u_i)_j | \vspace{.05in}\\
       				& \hspace{-.1in} \geq |\alpha_{\sigma(k)}| . | \sum_{j\in S^k} (u_k)_j | - \sum_{i \neq k} \left(| \alpha_{\sigma(i)} | . \sum_{j\in S^k} (u_i)_j \right) \vspace{.05in}\\
       				& \hspace{-.1in} \geq |\alpha_{\sigma(k)}| . | \sum_{j\in S^k} (u_k)_j | - \sum_{i \neq k} \left(| \alpha_{\sigma(i)} | . \sum_{j \not\in S^i} (u_i)_j \right) \vspace{.05in}\\
       				& \hspace{-.1in}\geq \alpha (1-2/(2N+1)) - \alpha (c-1).2/(2N+1) = \alpha(2(N-c)+1)/(2N+1) \vspace{.05in}\\
					& \hspace{-.1in} > 0,
    			\end{array}
    		\label{w2}
 		\end{equation}
which is a contradiction. Thus, we must have $\alpha = 0$, which means $\alpha_i = 0$, $\forall i$, $1 \leq i \leq c$. This proves the lemma.
	\end{proof}

	\begin{theorem}
  	\label{rank=c}
  		$\Rank(A)$ is equal to $c$, i.e, the constant value of $c_{\tau}$, $\tau \geq 0$, where $c_{\tau}$ is the number of vertices of limiting polytope $C_{\tau}$.
	\end{theorem}

	\begin{proof}
  		Theorem \ref{rank=c} is an immediate result of Lemmas \ref{space-generated} and \ref{independence}.
	\end{proof}

	Combining Theorems \ref{rank=EGC} and \ref{rank=c} result in the following corollary.

	\begin{corollary}
  		The size of the smallest EGC of a network with dynamics (\ref{mc}) is $c$.
  	\label{special agents}
	\end{corollary}

	\begin{lemma}
  		$c$ is less than or equal to the number of ergodicity classes.
  	\label{ergod}
	\end{lemma}
	\begin{proof}
  		Recall limiting polytope $C_0$ with vertices $v_1,\ldots,v_c$ from earlier in the section. Remember, from the proof of Lemma \ref{independence}, that for $i=1,\ldots,c$, there exists sequences $\{i_t\}$ of individuals such that $\Phi'_{i_t}(t,0)$ converges to $v_i$. Let:
		\begin{equation}
			\epsilon_2 = \frac{1}{3}\min \{ \| v_i - v_j \| \, | i,j \in \V,\, i \neq j \}.
		\label{305}
		\end{equation}
		By definition of ergodicity classes, there exists $T \geq 0$ such that for every $t \geq T$, for a fixed $\tau$, and for every $i,j$ in the same ergodicity class, we have:
		\begin{equation}
			\| \Phi'_i (t,\tau) - \Phi'_j(t,\tau) \| < \epsilon_2.
		\label{306}
		\end{equation}
		On the other hand, there exists $T' > 0$ such that for every $t \geq T'$, and $i = 1,\ldots,c$, we have:
		\begin{equation}
			\| \Phi'_{i_t}(t,0)-v_i \|<\epsilon_2.
		\label{307}
		\end{equation}
		Therefore, for every $t \geq T'$, and $i \neq j$, $1 \leq i,j \leq c$, we must have:
		\begin{equation}
			\begin{array}{ll}
				3\epsilon_2 	& \hspace{-.1in} \leq \| v_i - v_j \| \leq \| v_i - \Phi'_{i_t}(t,0) \| + \| \Phi'_{i_t}(t,0) - \Phi'_{j_t}(t,0) \| + \| \Phi'_{j_t}(t,0)-v_j \| \vspace{.05in}\\
								& \hspace{-.1in} < \epsilon_2 + \| \Phi'_{i_t}(t,0) - \Phi'_{j_t}(t,0) \| + \epsilon_2,
			\end{array}
		\label{308}
		\end{equation}
		where the first inequality above is a result of (\ref{305}), the second inequality is the triangle inequality, and the third inequality is a consequence of (\ref{307}). From (\ref{308}), we now have:
		\begin{equation}
			\| \Phi'_{i_t}(t,0) - \Phi'_{j_t}(t,0) \| > \epsilon_2, \, \forall t \geq T'.
		\label{309}
		\end{equation}
		Taking (\ref{306}) into account, from (\ref{309}) we conclude that $i_t$ and $j_t$ cannot be in the same ergodicity class for every $t \geq \max \{T,T'\}$. Thus, there are at least $c$ distinct ergodicity classes, and the lemma is proved.
	\end{proof}
	
	\begin{corollary}
	\label{upper bound ergodic}
  		For an arbitrary chain $\{A(t)\}$, $\Rank(A)$ is less than or equal to the number of ergodicity classes of $\{A(t)\}$.
	\end{corollary}
	
	\begin{corollary}
		For an opinion network with dynamics (\ref{mc}), the size of the smallest EGC is upper bounded by the number of ergodicity classes of $\{A(t)\}$.
	\end{corollary}
	
\begin{remark}
\label{remark-3}
  In case $\{A(t)\}$, the underlying chain of a network with dynamics (\ref{mc}), is class-ergodic, the occurrence of multiple consensus in the network is guaranteed, and the number of ergodic classes becomes equal to the number of consensus clusters. Yet this number may be larger than the size of the smallest EGC of the network. In other words, there may exist an EGC in which some of the consensus clusters have no representative. As a simple illustrative example, consider system (\ref{mc}) of three individuals with a fixed underlying chain:
  \begin{equation}
    A(t) = \begin{bmatrix}
      0 & 0 & 0\\
      1/3 & -1 & 2/3\\
      0 & 0 & 0
    \end{bmatrix} , \, \forall t \geq 0.
  \end{equation}
  We then have:
  \begin{equation}
    \lim_{t \rightarrow \infty} x(t) = \begin{bmatrix} x_1(t_0) \\ (x_1(t_0) + 2x_3(t_0))/3 \\ x_3(t_0) \end{bmatrix}.
  \end{equation}
  Notice also that for the corresponding state transition matrix we have:
\begin{equation}
	\lim_{t \rightarrow \infty} \Phi(t,\tau) =
		\begin{bmatrix}
  			1 	&	0	&	0	\\
    			1/3 &	0	&	2/3	\\
			0	&	0	&	1
		\end{bmatrix}, \, \forall \tau \geq 0.
\end{equation}
  Therefore, each individual forms a consensus cluster, i.e., there are three consensus clusters. However, subgroup $\{ 1,3 \}$ with size two, is an EGC of the network. In other words, starting at an arbitrary initial time $t_0 \geq 0$, irrespective of the initial opinion of individual 2, an agreement on value $x^*$ is achieved if individuals 1 and 3 initialize their opinions at $x^*$.
\end{remark}


\section{Lower Bounds on the $\text{Rank}$ of chains}
\label{lower bounds}

In this section, we clarify how the underlying chain of a network with dynamics (\ref{mc}) imposes lower bounds on the size of its smallest EGC, which is equal to $\Rank(A)$. We recall the following definition from \cite{Hend:11,Bolouki:12b}.
\begin{definition}
  The unbounded interactions graph of a chain $\{A(t)\}$, $\h_1(\V,\e_1)$, is a fixed directed graph such that for every distinct nodes $i,j \in \V$, $(i,j) \in \e_1$ if and only if:
  \begin{equation}
    \int_{0}^{\infty} a_{ji}(t) dt= \infty.
  \end{equation}
  In other words, a link is drawn from $i$ to $j$ if the total influence of individual $i$ on individual $j$ is unbounded over the infinite time interval.
\end{definition}

\begin{definition}
	A subset $\s' \subset \V$ is called a \textit{s-root} of $\h_1(\V,\e_1)$ if for every node $i \in \V$, we have $i \in \s'$ or there exists $j \in \s'$ such that $i$ is reachable from $j$.
\end{definition}

\begin{theorem}
\label{lower-2}
  Let $\h_1(\V,\e_1)$ be the unbounded interaction graph associated with chain $\{A(t)\}$. Then, $\Rank(A)$ is greater than or equal to the size of the smallest s-root of $\h_1(\V,\e_1)$.
\end{theorem}
\begin{proof}
  Form a chain $\{B(t)\}$ from chain $\{A(t)\}$ by eliminating all influences that individual $i \in \V$ gets from individual $j \in \V$ if $(j,i) \not\in \e_1$. More specifically, for every $i \neq j \in \V$ and $t \geq 0$, we have:
  \begin{equation}
  	b_{ij}(t) =
	\begin{cases}
		a_{ij}(t)	& \text{ if } (j,i) \in \e_1\\
		0		& \text{ if } (j,i) \not\in \e_1
	\end{cases}
  \end{equation}
  and $b_{ii}(t) = -\sum_{j \neq i} b_{ij}(t)$, for every $i \in \V$ and $t \geq 0$. Since chain $\{B(t)\}$ is an $l_1$-approximation of chain $\{A(t)\}$, from Lemma \ref{rank and approximation}, the two chains share the same rank. Notice also that the two chains share the same unbounded interactions graph. Thus, it suffices to prove Theorem \ref{lower-2} for chain $\{B(t)\}$. Consider an opinion network with underlying chain $\{B(t)\}$:
	\begin{equation}
		\dot{y}(t) = B(t)y(t), \, t \geq t_0,
	\label{mcc}
	\end{equation}
where $y(t) \in \mathbb{R}^N$ is the vector of opinions. Since $\Rank(B)$ is the size of the smallest EGC of the network with dynamics (\ref{mcc}), it is sufficient to show that every EGC of the network with dynamics (\ref{mcc}) is a s-root of $\h_1$. Assume, on the contrary, that subset $\s \subset \V$ is an EGC which is not a s-root of $\h_1$. Define:
\begin{equation}
	n(\s) ~\triangleq~ \s ~\cup~ \{i ~|~ i \in \V,\, \exists j \in \s : i \text{ is reachable from } j \text{ in } \h_1\}
\end{equation}
Since $\s$ is not a s-root, $n(\s) \subsetneq \V$. From the definition of $n(\s)$, it is easy to see that there is no link from $n(\s)$ to $\V \backslash n(\s)$ in $\h_1$. According to the way that chain $\{B(t)\}$ was constructed, this means that $n(\s)$ has zero influence on $\V \backslash n(\s)$ at any time instant. Thus, since $\s \subset n(\s)$, individuals in $\s$ cannot, in general, lead individuals in $\mathcal{V} \backslash n(\s)$ to agreeing on an arbitrary value $x^*$. For instance, given a desired consensus value $x^*$, if the opinions of individuals in $\mathcal{V} \backslash n(\s)$ are all initialized at value $x^* +1$, they will never change, and consequently, they will never converge to $x^*$. Thus, $\s$ is not an EGC, which completes the proof.
\end{proof}

An important special case of Theorem \ref{lower-2} is described in the following. Let us first define the continuous time counterpart of the \textit{infinite flow graph} of a chain according to \cite{Touri:10a}.
\begin{definition}
  The infinite flow graph $\h_2(\V,\e_2)$ of a given chain $\{A(t)\}$, is an undirected graph formed as follows: for two distinct nodes $i,j \in \V$, draw a link between $i$ and $j$ in $\h_2$, if and only if:
  \begin{equation}
    \int_{0}^{\infty} (a_{ij}(t)+a_{ji}(t))dt = \infty
  \end{equation}
\end{definition}
We now have the following lower bound on the $\text{rank}$ of a chain which is a special case of Theorem \ref{lower-2}.
\begin{corollary}
  $\Rank(A)$ is greater than or equal to the number of connected components of the infinite flow graph associated with $\{A(t)\}$.
  \label{lower-1}
\end{corollary}


\section{Rank of Time-Invariant (TI) chains}
\label{rank of TI chains}

Let $\{ A(t) \}$ be a TI chain, i.e., $A(t) = \hat A$, $\forall t \geq 0$, where $\hat A$ is a fixed matrix with the property that each of its rows adds up to zero and its off-diagonal elements are non-negative. Assume that $\rank(\hat{A})$ and $\nullity(\hat{A})$ represent the rank and the nullity of $\hat{A}$. Notice that $\mathrm{roman}$ style is used for matrix operators as opposed to the chain operators so as to avoid any ambiguity. For state transition matrix $\Phi(t,\tau)$ associated with TI chain $\{\hat{A}\}$, we have:
\begin{equation}
  \Phi(t,\tau) = e^{\hat A (t-\tau)},\, t \geq \tau \geq 0.
\end{equation}
Note that $\hat A$ is marginally stable and has all negative eigenvalues but one eigenvalue zero with algebraic multiplicity $\nullity(\hat A)$. Thus, $\lim_{t-\tau \rightarrow \infty}\Phi(t,\tau)$ exists, and the limit has eigenvalue zero with algebraic multiplicity $\rank(\hat A)$ and eigenvalue one with algebraic multiplicity $\nullity(\hat A)$. Hence:
\begin{equation}
  \Rank(A) = \nullity(\hat A).
\end{equation}
Employing a graph theoretic approach, treating $\hat A$ as the Laplacian of its associated \textit{weighted directed} graph, $\nullity(\hat A)$ represents the size of the smallest s-root of the graph (see Fig. \ref{Fig3}).

\begin{figure}[h]
  \begin{center}
  \captionsetup{justification=centering}
    \includegraphics[scale=1]{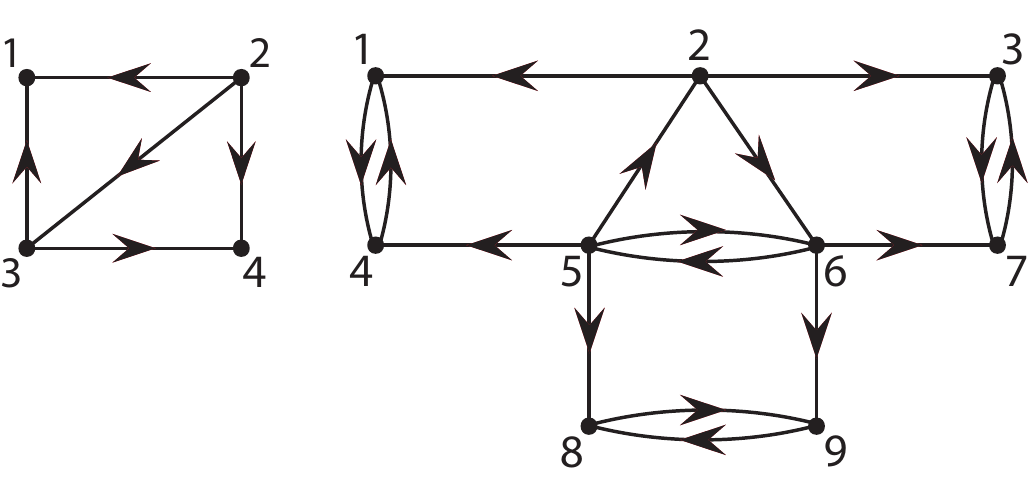}
    \caption{Unweighted underlying graph of two TI linear algorithms. $\{1,4\}$ (left) and $\{1,3,8\}$ (right) are the smallest s-roots.}
    \label{Fig3}
  \end{center}
\end{figure}

Since an unweighted version of the graph described above serves as the unbounded interactions graph associated with TI chain $\{ A(t) \}$, $A(t) = \hat A$, $\forall t \geq 0$, we have the following corollary.

\begin{corollary}
  For a TI chain $\{A(t)\}$, the lower bound provided in Theorem \ref{lower-2} is achieved. More specifically, $\Rank(A)$ is size of the smallest s-root of the unbounded interactions graph associated with $\{A(t)\}$.
\end{corollary}

Remember that any TI chain $\{A(t)\}$ is class-ergodic and the number of ergodic classes provides an upper bound for $\Rank(A)$ according to Corollary \ref{upper bound ergodic}. For example, for the underlying graphs depicted in Fig. \ref{Fig3}, the number of ergodic classes are 4 (left) and 6 (right).

 The graph interpretation of the notion of $\text{rank}$ explains the following two properties:
\begin{enumerate}[(i)]
  \item For any TI chain $\{A(t)\}$ and $\alpha > 0$:
    \begin{equation}
      \Rank(\{\alpha A(t)\}) = \Rank(\{A(t)\}).
    \end{equation}
  \item For any two TI chains $\{A(t)\}$ and $\{B(t)\}$,
    \begin{equation}
      \Rank(\{A(t)+B(t)\}) \leq \min \Big\{ \Rank(\{A(t)\}),\Rank(\{B(t)\}) \Big\}.
    \end{equation}
\end{enumerate}
\begin{remark}
While Statement (i) seems to hold for any time-varying chain $\{A(t)\}$ as well, there exist time-varying chains $\{A(t)\}$ and $\{B(t)\}$ that do not satisfy Statement (ii). This means that more interactions between agents may surprisingly increase the size of the smallest EGC of a network. The following is an example; let:
\begin{equation}
  A(t) = \begin{bmatrix} -1 & 1 & 0 \\ 0 & 0 & 0 \\ 0 & 0 & 0 \end{bmatrix} \text{ if } t \in [2^{2k}-1, 2^{2k}), \, k \in \mathbb{N},
\end{equation}
and,
\begin{equation}
  A(t) =  \begin{bmatrix} 0 & 0 & 0 \\ 0 & -1 & 1 \\ 0 & 0 & 0 \end{bmatrix} \text{ if } t \in [2^{2k}, 2^{2k+1}-1), \, k \in \mathbb{N},
\end{equation}
and $A(t)=\textbf{0}_{3 \times 3}$ elsewhere. Let also:
\begin{equation}
  B(t) = \begin{bmatrix} 0 & 0 & 0 \\ 0 & 0 & 0 \\ 0 & 1 & -1 \end{bmatrix} \text{ if } t \in [2^{2k+1}-1, 2^{2k+1}), \, k \in \mathbb{N},
\end{equation}
and,
\begin{equation}
  B(t) =  \begin{bmatrix} 0 & 0 & 0 \\ 1 & -1 & 0 \\ 0 & 0 & 0 \end{bmatrix} \text{ if } t \in [2^{2k+1}, 2^{2k+2}-1), \, k \in \mathbb{N},
\end{equation}
and $B(t)=\textbf{0}_{3 \times 3}$ elsewhere. Note that at every time instant either $A(t)$ or $B(t)$ is $\textbf{0}_{3 \times 3}$. It is easy to see that both $\{A(t)\}$ and $\{B(t)\}$ are ergodic chains. More specifically, for every $\tau \geq 0$, we have:
\begin{equation}
  \lim_{t \rightarrow \infty} \Phi_A(t,\tau) = \begin{bmatrix} 0 & 0 & 1 \end{bmatrix} \begin{bmatrix} 1 & 1 & 1 \end{bmatrix}', 
\end{equation}
and,
\begin{equation}
  \lim_{t \rightarrow \infty} \Phi_B(t,\tau) = \begin{bmatrix} 1 & 0 & 0 \end{bmatrix} \begin{bmatrix} 1 & 1 & 1 \end{bmatrix}'.
\end{equation}
Therefore, $\Rank(A) = \Rank(B) = 1$. However, one can show that $\Rank(\{A(t) + B(t) \}) =2$. More precisely, subgroup $\{1,3\}$ forms the smallest EGC of the network with underlying chain $\{A(t) + B(t)\}$.
\end{remark}

\section{Rank of chains in Class $\mathcal{P}^*$}
\label{Class P*}

From the fundamental work \cite{Kolmo:36}, it is known that for every state transition matrix $\Phi(t,\tau)$, $t \geq \tau \geq 0$, associated with a chain $\{A(t)\}$, there exists a sequence of stochastic row vectors $\{ \pi(t) \}$, called an \textit{absolute probability sequence}, such that:
\begin{equation}
  \pi(\tau) = \pi(t) \Phi(t,\tau), \, \forall t,\tau, \, t \geq \tau \geq 0.
\end{equation}
Remember that by a stochastic vector, we mean a vector with elements adding up to 1. We may now extend \cite[Definition 3]{Touri:11c} to the continuous time case in the following.
\begin{definition}
 A chain $\{A(t)\}$ is said to be in Class $\mathcal{P}^*$ if its associated state transition matrix $\Phi(t,\tau)$, $t \geq \tau \geq 0$ admits an absolute probability sequence $\{ \pi(t) \}$ such that for some constant $p^* > 0$:
 \begin{equation}
   \pi(t) > p^*, \, \forall t \geq 0.
 \end{equation}
\end{definition}
It is possible to characterize chains of Class $\mathcal{P}^*$ more concretely. To do so, we first state the following lemma.
\begin{lemma}
  For every $j \in \V$,
	\begin{equation}
	  \pi_j(\tau) \leq \inf \left\{\sum_{i \in \V}\Phi_{i,j}(t,\tau) \, | \, t \geq \tau \right\}.
	\end{equation}
	\label{inf}
\end{lemma}
\begin{proof}
  Obvious, since for every $t \geq \tau$:
	\begin{equation}
	  \pi_j(\tau) = \pi(t) \Phi^j(t,\tau) = \sum_{i \in \V} \pi_i(t) \Phi_{i,j}(t,\tau) \leq \sum_{i \in \V}\Phi_{i,j}(t,\tau).
	\end{equation}
\end{proof}
We now have the following lemma that provides an alternative definition of chains in Class $\mathcal{P}^*$.
\begin{lemma}
  A chain $\{A(t)\}$ is in Class $\mathcal{P}^*$ if and only if for its state transition matrix $\Phi(t,\tau)$, $t \geq \tau \geq 0$, we have:
	\begin{equation}
	  \inf_{t , \tau} \left\{\sum_{i \in \V}\Phi_{i,j}(t,\tau) \, | \, t \geq \tau \geq 0 \right\} > 0, \, \forall j \in \V.
	\end{equation}
	\label{pstar}
\end{lemma}
\begin{proof}
  The ``only if'' part is an immediate result of Lemma \ref{inf}, and the ``if'' part is a result of the way an absolute probability sequence can be obtained in \cite{Kolmo:36} by always choosing to initialize agent probabilities on finite intervals with a uniform distribution.
\end{proof}
Lemma \ref{pstar} roughly implies that the underlying chain of a system is in Class $\mathcal{P}^*$, if and only if the opinion of any individual, at any time, continues to have influence on the formation of individuals' opinions at all future times. We now state a theorem on the class-ergodicity of chains in Class $\mathcal{P}^*$ (see \cite[Theorem 6]{Bolouki:14a}).

\begin{theorem}
\label{cont}
  Every chain $\{ A(t) \}$ in Class $\mathcal{P}^*$ is class-ergodic. Furthermore, the number of ergodic classes is equal to the number of connected components of the infinite flow graph of chain $\{ A(t) \}$.
\end{theorem}

Theorem \ref{cont} implies that if chain $\{ A(t) \}$ is in Class $\mathcal{P}^*$, the upper bound provided for its $\text{rank}$ in Corollary \ref{upper bound ergodic} is equal the lower bound provided in Corollary \ref{lower-1}. Therefore, both bounds become equal to $\Rank(A)$.
 \begin{corollary}
   The $\text{rank}$ of a chain in Class $\mathcal{P}^*$ is determined by the number of connected components of the infinite flow graph associated with the chain.
 \end{corollary}


\section{Full-Rank chains}
\label{full-rank chains}

 One can characterize chains with maximum possible $\text{rank}$ as the following.
\begin{theorem}
\label{full-rank thm}
  A chain $\{ A(t) \}$ is \textit{full-rank}, i.e., $\Rank(A)=N$ if and only if $\{ A(t) \}$ is an $l_1$-approximation of the neutral chain, i.e., the chain of matrix $\textbf{0}_{N \times N}$.
\end{theorem}
\begin{proof}
  The sufficiency is immediately implied using Lemma \ref{rank and approximation} and taking into account that the neutral chain is full-rank. To prove the necessity, assume that $\{A(t)\}$ is full-rank. We may now once again take advantage of our geometric framework developed in Section \ref{a geometric framework} based on the associated state transition matrix. Recall that $c$ is defined by the number of vertices of limiting polytope $C_{\tau}$ for an arbitrary $\tau \geq 0$. Since $\Rank(A) = c$, we conclude that $c=N$. Letting $v_1,\ldots,v_N$ be the $N$ vertices of $C_0$, for a permutation $\sigma$ over $\{1,\ldots,N\}$, we must have:
  \begin{equation}
    \lim_{t \rightarrow \infty} \Phi'(t,0) = \begin{bmatrix} v_{\sigma(1)} | \cdots | v_{\sigma(N)} \end{bmatrix},
    \label{fulls}
  \end{equation}
  since each column of $\Phi'(t,0)$ is a continuous function of $t$ such that its distance from $\{ v_1,\ldots,v_N \}$ vanishes as $t$ grows large. Recalling:
  \begin{equation}
    \Phi(t,0) = \Phi(t,\tau) \Phi(\tau,0), \, \forall t \geq \tau \geq 0,
  \end{equation}
  and taking into account that, based on Lemma \ref{independence}, the columns of the RHS of relation (\ref{fulls}) are linearly independent stochastic vectors, for a sufficiently large $T \geq 0$, $\Phi(t,\tau)$ is arbitrarily close to the $N \times N$ identity matrix for every $t \geq \tau \geq T$. In particular, $\Phi(t,\tau)$ has positive diagonal elements (well away from zero) for every $t \geq \tau \geq T$. Form chain $\{B(t)\}$ from $\{A(t)\}$ by eliminating all interactions between individuals over time interval $[0,T)$. Then, the state transition matrix associated with chain $\{ B(t) \}$ has positive diagonal elements all the times. Recalling Lemma \ref{pstar}, we conclude that chain $\{ B(t) \}$ is in Class $\mathcal{P}^*$. On the other hand, chain $\{ B(t) \}$ is an $l_1$-approximation of chain $\{ A(t) \}$ due to boundedness of interactions over time interval $[0,T)$. Consequently, $\Rank(B) = \Rank(A) = N$. Theorem \ref{cont} now implies that $\Rank(B)=N$ is the number of connected components of the infinite flow graph associated with chain $\{B(t)\}$. This completes the proof since the two chains share the same infinite flow graph.
\end{proof}
Assume that the infinite flow graph of chain $\{ A(t) \}$, i.e., $\h_2(\V,\e_2)$, has $h_2$ connected components. Form chain $\{B(t)\}$, which is an $l_1$-approximation of $\{ A(t) \}$ by eliminating all interactions between distinct connected components. Since the subchain corresponding to each connected component is full-rank if and only if it contains a single node, the following proposition follows from Lemma \ref{rank and approximation}, that provides an upper bound for $\Rank(A)$.
\begin{proposition}
\label{upper-2}
  Let $\{A(t)\}$ be a time-varying chain with infinite flow graph $\h_2$. Then:
  \begin{equation}
    \Rank(A) \leq N-h'_2,
  \end{equation}
  where $h'_2$ is the number of connected components of $\h_2$ containing two or more nodes.
\end{proposition}


\section{Discrete Time Analysis}
\label{discrete time analysis}

	In this section, we turn our attention to the case in which the opinions of the individuals are updated at discrete time instants. Our aim is to characterize EGC's in a network for the discrete time case. To this aim, we adopt, with a slight modification, the same approach followed in the continuous time case, i.e., an approach based on the notion of rank. After we define the $\text{rank}$ of a discrete time chain, we carry out the discrete time counterpart of our statements in Sections \ref{problem setup}--\ref{full-rank chains}.

	Remember that in this section, time variables $t,\tau,t_0$, etc. refer to the discrete time indices. Let $\{A(t)\}_{t \geq 0}$ be a time-varying chain of row-stochastic square matrices of size $N$. A row-stochastic matrix, or simply stochastic matrix, is a matrix with non-negative elements and the property that its each row elements sum up to 1 . Discrete time chains of matrices, that we deal with in this paper, are assumed to be chains of stochastic matrices. Indeed, $A(t)$ can be viewed as the transition matrices of a time inhomogeneous Markov chain. Let dynamics of an opinion network be described by the following discrete time distributed averaging algorithm:
	\begin{equation}
		x(t+1) = A(t) x(t), \, t \geq t_0,
	\label{md}
	\end{equation}
where $t_0 \geq 0$ is the initial time, $x(t) \in \mathbb{R}^N$ is the vector of opinions at each time instant $t \geq t_0$, and chain $\{A(t)\}_{t \geq 0}$, or simply $\{A(t)\}$, is the underlying chain of the network.

	The notion of EGC in a network of individuals with discrete time dynamics (\ref{md}) is defined consistently with Definition \ref{EGC def}. More specifically, for an opinion network with dynamics (\ref{md}), an EGC refers to a subgroup of individuals who are able to lead the whole group to asymptotically agreement on any desired value by cooperatively and properly choosing their own initial opinions, based on an awareness of underlying chain $\{A(t)\}$ as well as the initial opinions of the rest of individuals. Notice that Lemma \ref{zero is enough}, with a similar proof, also holds for a network with dynamics (\ref{md}). In the following, by extending the notions of null space, nullity, and rank to discrete time chains, we exploit the relationship between the characterization of an EGC in a network, size of the smallest EGC, and properties of the underlying chain of the network.

	For the sake of notational consistency, let $\Phi(t,\tau)$, $t \geq \tau \geq 0$, be the state transition matrix associated with discrete time chain $\{A(t)\}$. State transition matrix $\Phi(t,\tau)$ satisfies relation (\ref{mgeneral}). we also have:
\begin{equation}
\label{phi discrete}
  \Phi(t,\tau) = A(t-1) \cdots A(\tau), \, \forall t > \tau \geq 0,
\end{equation}
and $\Phi(t,t) = I_{N \times N}$, $\forall t \geq 0$. Define the null space of discrete time chain $\{ A(t) \}$ at an arbitrary time instant $\tau \geq 0$, $\Null_{\tau}(A)$, consistently with its continuous time version, i.e., Definition \ref{null space}. $\Null_{\tau}(A)$, $\tau \geq 0$, is again a vector space. However, since the state transition matrix in the discrete time case may be singular at times, unlike the continuous time case, the dimension of $\Null_{\tau}(A)$, denoted by $\dim (\Null_{\tau}(A))$, can vary as $\tau$ grows. However, it is not difficult to show that $\dim (\Null_{\tau}(A))$ is non-increasing with respect to $\tau$. We now have the following theorem on the size of the smallest EGC of a network with dynamics (\ref{md}). The proof is eliminated as it is similar to the proof of Theorem \ref{rank=EGC}.

\begin{theorem}
\label{EGC D}
	For an opinion network with dynamics (\ref{md}), the size of the smallest EGC is $N - \dim (\Null_{t_0}(A))$.
\end{theorem}

Since $\dim (\Null_{\tau}(A))$ is non-increasing with respect to $\tau$, from Theorem \ref{EGC D}, we conclude that initializing the network with dynamics (\ref{md}) at a later time results in a greater or equal size of its smallest EGC. Notice now that $\dim (\Null_{t_0}(A))$ is an integer-valued operator bounded below by zero. Thus, $\dim (\Null_{\tau}(A))$ becomes constant after a \textit{finite} time. Define the nullity of chain $\{A(t)\}$, $\Nullity(A)$, by that constant:
\begin{equation}
  \Nullity(A) \triangleq \lim_{\tau \rightarrow \infty} \dim (\Null_{\tau}(A)).
  \label{lim null}
\end{equation}
Define now the rank of chain $\{A(t)\}$, $\Rank(A)$, as in continuous time, by $\Rank(A) = N - \text{null}(A)$. The following corollary, to be viewed as the discrete time counterpart of Theorem \ref{rank=EGC}, is an immediate result of Theorem \ref{EGC D} and the definition of $\Rank(A)$.

\begin{corollary}
If a network with dynamics (\ref{md}) is initialized at a sufficiently large time, the size of its smallest EGC is $\Rank(A)$, where a sufficiently large time refers to some time after the RHS of (\ref{lim null}) has converged.
\end{corollary}

In the rest of this section, we focus on the notion of rank of a chain. We recall the definition of $l_1$-approximation of a discrete time chain from \cite{Touri:10b}.

\begin{definition}
  Chain $\{A(t)\}$ is said to be an \textit{$l_1$-approximation} of chain $\{B(t)\}$ if:
  \begin{equation}
    \sum_{t=0}^{\infty} \| A(t)-B(t) \| < \infty,
  \end{equation}
  where for convenience only, the norm refers to the \textit{max norm}, i.e., the maximum of the absolute values of the matrix elements.
\end{definition}

It can be shown that, $\text{rank}$, as we defined it for the discrete time case, is invariant under an $l_1$-approximation, i.e., Lemma \ref{rank and approximation} holds for the discrete time case as well.


\subsection{Rank via Sonin Decomposition-Separation Theorem}

We aim to address in this subsection, the rank of a discrete time chain of stochastic matrices via an approach based on the Sonin D-S Theorem \cite{Sonin:08,Bol:13}. Some preliminaries are required first. According to \cite{Blackwell:45} as reported in \cite{Sonin:08}, the definition of jet will be recalled. It plays a crucial role in our discrete time arguments.

\begin{definition}
  Given the set of individuals $\V = \{1,\ldots,N\}$, a \textit{jet} $J$ in $\V$ is a sequence $\{ J(t) \}$ of subsets of $\V$. A jet $J$ in $\V$ is called a \textit{proper} jet if $\emptyset \neq J(t) \subsetneq \V$, $\forall t \geq 0$. Complement of jet $J=\{J(t)\}$ in $\V$, denoted by $\bar{J}$ is also a jet in $\V$ expressed by sequence $\{ \V \backslash J(t) \}$. For a fixed subset $S \subset \V$, jet $S$ refers to a jet which is equal to $S$ at all time instants.
\label{jet-def}
\end{definition}

\begin{definition}
  A tuple of jets $(J^1,\ldots,J^c)$ is a \textit{jet-partition} of $\V$, if $(J^1(t),\ldots,J^c(t))$ forms a partition of $\V$ for every $t \geq 0$.
\end{definition}

Consider a multi-agent system with states evolving according to linear algorithm (\ref{md}). Based on the work \cite{Kolmo:36}, we know that discrete time chain $\{ A(t) \}$ admits an absolute probability sequence $\{\pi(t)\}$ which propagates backwards in time:
\begin{equation}
  \pi'(t+1)A(t) = \pi'(t), \forall t \geq 0.
\end{equation}
From chain $\{ A(t) \}$, construct chain $\{ P(t) \}$ of stochastic matrices satisfying:
\begin{equation}
  \pi_i(t) p_{ij}(t) = \pi_j(t+1) a_{ji}(t), \forall i,j \in \V, \forall t \geq 0.
\end{equation}
More specifically, if $\pi_i(t) \neq 0$, then set:
\begin{equation}
  p_{ij}(t) = \pi_j(t+1) a_{ji}(t) / \pi_i(t),
\end{equation}
while if $\pi_i(t) = 0$ for some $i \in \V$ and $t \geq 0$, choose non-negative $p_{ij}(t)$'s arbitrarily such that:
\begin{equation}
  \sum_{j=1}^N p_{ij}(t) = 1.
\label{sto}
\end{equation}
Note that in the former case ($\pi_i(t) \neq 0$), (\ref{sto}) is automatically satisfied, implying that $P(t)$ is a stochastic matrix for every $t \geq 0$. It is easy to see that:
\begin{equation}
  \pi'(t) P(t) = \pi'(t+1), \forall t \geq 0,
\end{equation}
indicating that $\{ \pi(t) \}$ can now be viewed as a non homogeneous \textit{forward} propagating Markov chain.
\begin{definition}
  Let the \textit{total flow} between two arbitrary jets $J^s$ and $J^k$ in $\V$ over the infinite time interval, denoted by $V(J^s,J^k)$, be defined as:
\begin{equation}
  V(J^s,J^k) \triangleq \sum_{t=0}^{\infty} \left[\sum_{i \in J^k(t)}\sum_{j \in J^s(t+1)} r_{ij}(t) + \sum_{i \in J^s(t)}\sum_{j \in J^k(t+1)} r_{ij}(t) \right],
\label{V}
\end{equation}
where
\begin{equation}
  r_{ij}(t) = \pi_i(t)p_{ij}(t) = \pi_j(t+1)a_{ji}(t).
  \label{r-p-a}
\end{equation}
\end{definition}

From a Markov chain point of view, value $r_{ij}(t)$ can be interpreted as the absolute joint probability of being in state $i$ at time $t$ and state $j$ at time $t+1$.

\begin{theorem} \textit{(Sonin D-S Theorem)}
  There exists an integer $c$, $1 \leq c \leq N$, and a decomposition of $\V$ into jet-partition $(J^0,J^1,\ldots,J^c)$, $J^k = \{ J^k(t) \}$, such that irrespective of the particular time or values at which $x_i$'s are initialized,
  \begin{enumerate}[(i)]
    \item For every $k$, $1 \leq k \leq c$, there exist real constants $\pi^*_k$ and $x^*_k$, such that:
	\begin{equation}
	  \lim_{t \rightarrow \infty} \sum_{i \in J^k(t)} \pi_i(t)=\pi^*_k,
	\end{equation}
	and:
	\begin{equation}
	  \lim_{t \rightarrow \infty} x_{i_t}(t)=x^*_k,
	\end{equation}
	for every sequence $\{ i_t \}$, $i_t \in J^k(t)$. Furthermore, $\lim_{t \rightarrow \infty}\sum_{i \in J^0(t)} \pi_i(t) = 0$.
    \item For every distinct $k,s$, $0 \leq k,s \leq c$: $V(J^k,J^s) < \infty$.
    \item This decomposition is unique up to jets $\{ J(t) \}$ such that for any $\{ \pi(t) \}$ we have:
	\begin{equation}
	  \lim_{t \rightarrow \infty}\sum_{i \in J(t)} \pi_i(t) = 0,
	\end{equation}
      and:
      \begin{equation}
        V(J,\V\backslash J) < \infty.
      \end{equation}
  \end{enumerate}
\end{theorem}

\begin{theorem}
\label{rank jets}
  The unique jet decomposition of $\V$ with respect to chain $\{ A(t) \}$ in the Sonin D-S Theorem, consists of jet $J^0$ and $\Rank(A)$ other jets.
\end{theorem}
\begin{proof}
  Theorem \ref{rank jets} is an immediate result of \cite[Remark 2]{Bolouki:14a} combined with Theorem \ref{geo-d}, that will be stated later in the paper.
\end{proof}


\subsection{A Geometric Interpretation}
\label{geo-int-discrete}

We developed, in Section \ref{a geometric framework}, a geometric framework, that interprets the $\text{rank}$ of the underlying chain of a network, based on the state transition matrix of the network, i.e, $\Phi(t,\tau)$. A similar argument can be made for the discrete time case, with the state transition matrix expressed as (\ref{phi discrete}). The only difference here is that $c_{\tau}$,  which is the number of vertices of limiting polytope $C_{\tau}$, is not invariant as $\tau$ grows. As a matter of fact, it can be shown that:
\begin{equation}
  c_{\tau} = N - \dim (\mathcal{N}_{\tau}(A)).
  \label{708}
\end{equation}
Therefore, $c_{\tau}$ is a non-decreasing function of $\tau$ and becomes constant after a finite time since it is bounded above by $N$. In correspondence to Theorem \ref{rank=c}, we have the following theorem:

\begin{theorem}
\label{geo-d}
  For the number of the vertices of limiting polytope $C_{\tau}$, $\tau \geq 0$, i.e., $c_{\tau}$:
\begin{equation}
  \lim_{\tau \rightarrow \infty}c_{\tau} = \Rank(A).
  \label{709}
\end{equation}
Consequently, there exist $t_0 \geq 0$ such that $c_{\tau}$ is equal to $\Rank(A)$ for every $\tau \geq t_0$.
\end{theorem}
\begin{proof}
  (\ref{709}) is easily obtained by taking the limit of both sides of (\ref{708}) as $t \rightarrow \infty$.
\end{proof}
Similar to the continuous time case, we define ergodicity classes of a discrete time chain as equivalence classes resulted by the relation of being weakly mutually ergodic (see Definition \ref{ergodicity classes}). It can be shown, similar to the proof of Lemma \ref{ergod}, that $c_{\tau}$ for every $\tau \geq 0$ is less than or equal to the number of ergodicity classes (note that ergodicity classes are defined irrespective of the initial time). This, together with Theorem \ref{geo-d}, implies that the number of ergodicity classes being an upper bound for the $\text{rank}$, i.e., Corollary \ref{upper bound ergodic}, also holds in the discrete time case.

\subsection{Lower Bounds}

We stated, in Theorem \ref{lower-2} and Corollary \ref{lower-1}, lower bounds on the rank of a continuous time chain. The discrete time counterparts of these theorems are subsumed through an approach employing the notion of jets.

\begin{definition}
  For a jet $J$ in $\V$, let $U_{in}(J)$ denote the total influence of $\bar{J}$ on $J$ over the infinite time interval:
  \begin{equation}
    U_{in}(J) = \sum_{t=0}^{\infty} \sum_{i \in J(t+1)} \sum_{j \not\in J(t)} a_{ij}(t).
  \end{equation}
\end{definition}

\begin{theorem}
\label{lower-d}
  For a discrete time chain $\{ A(t) \}$, $\Rank(A)$ is greater than or equal to the maximum number of disjoint jets, say $J$, each of which satisfying:
  \begin{equation}
    U_{in}(J) < \infty.
  \label{605}
  \end{equation}
\end{theorem}
\begin{proof}
  The proof of Theorem \ref{lower-d} is similar to that of Theorem \ref{lower-2}. For chain $\{A(t)\}$, let $J^1,\ldots,J^d$ be $d$ disjoint jets. Form a chain $\{B(t)\}$ from chain $\{A(t)\}$ by eliminating all interactions between any two distinct jets among $J^1,\ldots,J^d$ over the infinite time interval. Since $\{B(t)\}$ is an $l_1$-approximation of $\{A(t)\}$, the two chains share the same $\text{rank}$, as well as the same collections of disjoint jets. Therefore, it is sufficient to prove Theorem \ref{lower-d} for chain $\{B(t)\}$. Note that for chain $\{B(t)\}$, for every $s \neq k$, $1 \leq s,k \leq d$, we have:
\begin{equation}
  \sum_{t=0}^{\infty} \left[\sum_{i \in J^s(t+1)}\sum_{j \in J^k(t)} b_{ij}(t) + \sum_{i \in J^k(t+1)}\sum_{j \in J^s(t)} b_{ij}(t) \right] = 0.
\label{U is 0}
\end{equation}
We now consider an opinion network with underlying chain $\{B(t)\}$. Keeping Theorem \ref{rank=EGC} in mind, it suffices to show that the size of the smallest EGC of the opinion network defined over chain $\{B(t)\}$ is at least $d$. Consider a particular EGC of the opinion network defined over chain $\{B(t)\}$. By definition, that particular EGC is able to create global consensus under certain circumstances for infinitely many choices of initial time. Let $t_0 \geq 0$ be one of those infinitely many possible choices of initial time. Relation (\ref{U is 0}) means that for any jet among $J^1,\ldots,J^d$, say $J^s$, the opinions of individuals in $J^s(t)$, $\forall t \geq t_0$, only depend on the opinion of individuals in $J^s(t_0)$. Therefore, that particular EGC must contain at least one of the individuals in $J^s(t_0)$ or else it would have no control on the opinion of individuals in jet $J^s$ at any future time. Thus, the size of that particular EGC is greater than or equal to $d$, which is the number of disjoint jets $J^1,\ldots,J^d$. This proves the theorem.
\end{proof}

Theorem \ref{lower-d} would serve as the discrete time counterpart of Theorem \ref{lower-2}, if the choice of jets were limited to the time-invariant jets.


We skip the analysis of time-invariant discrete time chains, since it is no different from its continuous time counterpart.


\subsection{Rank of Discrete Time Chains in Class $\mathcal{P}^*$}

We, first, briefly discuss the limiting behavior of a discrete time chain $\{ A(t) \}$ in Class $\mathcal{P}^*$ from two viewpoints: (i) The Sonin D-S theorem; (ii) The geometric viewpoint. Given that $\{ A(t) \}$ belongs to Class $\mathcal{P}^*$, there is a representation of Sonin's jet decomposition without a $J^0$ jet. Therefore, each individual lies within $\cup_{k=1}^c J^k(t)$ for any $t \geq 0$, with $c$ being equal to $\Rank(A)$. Thus, the opinion of each individual stays arbitrarily close to set $\{ x^*_k\, |\, 1 \leq k \leq c\}$, with size $\Rank(A)$, as $t$ grows large. Considering now the geometric viewpoint, we focus on limiting polytopes $C_{\tau}$ as discussed in Section \ref{geo-int-discrete}. For the discrete time case, it was pointed out that the number of vertices of $C_{\tau}$ is non-decreasing and becomes constant past a finite time $t_0 \geq 0$, with $\Rank(A)$ being that constant. As proved in \cite{Bolouki:14a}, if $\{ A(t) \}$ is in Class $\mathcal{P}^*$, for every arbitrary fixed $\tau \geq t_0$, every column of $\Phi'(t,\tau)$ stays arbitrarily close to the $\Rank(A)$ vertices of $C_{\tau}$ as $t$ grows large. Since $x(t) = \Phi(t,\tau) x(\tau)$, each column $i$ of $\Phi'(t,\tau)$ (row $i$ of $\Phi(t,\tau)$) is in correspondence with the opinion of an individual $i$. Thus, columns of $\Phi'(t,\tau)$ staying arbitrary close to the $\Rank(A)$ vertices of $C_{\tau}$ as $t \rightarrow \infty$, leads to the same conclusion from the other point of view, that is the opinions staying arbitrary close to a set of $\Rank(A)$ (generally distinct) values. Thus, to sum up, although convergence of each individual's opinion is not guaranteed here unlike the continuous time case, there is a \textit{finite number} of accumulation points for the opinions over the infinite time interval, and that finite number is $\Rank(A)$.

Now reconsider jet-partition $(J^1,\ldots,J^c)$ in the Class $\mathcal{P}^*$ based jet-decomposition of the Sonin D-S Theorem. According to the Sonin D-S Theorem, for every two jets $J^k$ and $J^s$, we have:
\begin{equation}
  V(J^k,J^s) < \infty.
  \label{v-sonin}
\end{equation}
Recalling (\ref{r-p-a}) and taking into account that $\pi_j(t+1)$ in (\ref{r-p-a}) is greater than or equal to some $p^* > 0$ since chain $\{A(t)\}$ has been assumed to be in Class $\mathcal{P}^*$, inequality (\ref{v-sonin}) implies that the total interaction between any two jets $J^k$ and $J^s$ is finite over the infinite time interval, i.e.,
\begin{equation}
  \sum_{t=0}^{\infty} \left[\sum_{i \in J^{k+1}(t)}\sum_{j \in J^s(t)} a_{ij}(t) + \sum_{i \in J^{s+1}(t)}\sum_{j \in J^k(t)} a_{ij}(t) \right] < \infty.
\label{total-int}
\end{equation}
Fix an arbitrary $k$, and consider the set of inequalities obtained as $s \neq k$ goes from 1 to $c$ in (\ref{total-int}). Adding the $c-1$ obtained inequalities of type (\ref{total-int}), and noting that $J^1,\ldots,J^c$ is a jet partition of $\V$, we conclude that the total interaction between $J^k$ and $\bar{J}^k$, and in particular the total influence of $\bar{J}^k$ over $J^k$, is also finite over the infinite time interval. Therefore, for each of disjoint jets $J^1,\ldots,J^c$, say $J^k$, $V_{in}(J^k) < \infty$ (see (\ref{605})). Thus, recalling $\Rank(A) = c$, we conclude that the lower bound provided in Theorem \ref{lower-d} is achieved for discrete time chains in Class $\mathcal{P}^*$.


\subsection{Full-Rank Chains}

One characterizes full-rank discrete time chains according to the following theorem.

\begin{theorem}
\label{full-rank d}
  A discrete time chain $\{ A(t) \}$ is full-rank, i.e., $\Rank(A)=N$ if and only if $\{ A(t) \}$ is an $l_1$-approximation of a permutation chain, i.e., a chain of permutation matrices.
\end{theorem}

\begin{proof}
  The proof of Theorem \ref{full-rank d}, which is the discrete time version of Theorem \ref{full-rank thm}, is omitted since the proofs of the two theorems are very similar.
\end{proof}


\section{Conclusion}
\label{conclusion}

We considered a network of multiple individuals with opinions updated via a general time-varying continuous or discrete time linear algorithm. The notion of EGC, an acronym associated with \'Eminence Grise Coalition, in the network was defined as follows. Given the time that network starts to update, an EGC is a subgroup of individuals who, cooperatively, can manage to create a global consensus on any desired opinion in the network only by adequately setting their initial opinions assuming that they are aware of the underlying chain of the network as well as the rest of individuals initial opinions. The size of the smallest EGC can be treated as a characteristic of the underlying update chain of the network. We then introduced an extension of the notion of rank, from an individual matrix related notion to one related to a Markov chain in continuous or discrete time. A key result is that the $\text{rank}$ of the underlying chain of a network is also the size of its smallest EGC in the continuous time case. The same holds in the discrete time case provided the initial time is ``sufficiently large'' in a sense made precise in the paper. Geometrically, and associated with the chain, one can define a monotone decreasing convex hulls (polytopes) generated by an underlying  sequence of vertices. The $\text{rank}$ of the chain is the limiting number of linearly independent vertices in the sequence of polytopes, which is reached in finite time.

The continuous time case is peculiar in the sense that the $\text{rank}$ (number of linearly independent vertices) of the elements of the polytopic sequence remains constant, while it is monotonically increasing in the discrete time case. This, in turn, makes consensus behavior somewhat simpler in continuous time than in discrete time. A collection of upper and lower bounds on the $\text{rank}$ was also established. These two bounds are shown to be equal to the $\text{rank}$ for both time invariant chains (possibly not in Class $\mathcal{P}^*$), as well as for Class $\mathcal{P}^*$ chains in the time inhomogeneous case.

From a practical standpoint, this work establishes the rather intuitive result that the less ``natural'' dissension exists in an opinion network, the easier it is to steer the network towards global consensus. In cases where an ``average'' amount of natural dissonance exists, then the theory points at the need to minimally ``infiltrate'' identifiable dissenting clusters and work from the inside so to speak to steer the global opinion to a consensus. Success in doing so hinges on an ability to enlist key agents cooperation given that they must act as a ``grand coalition'' of key agents. This in turn opens the door to games over opinion networks whereby key agents might choose to break up into smaller coalitions and work towards conflicting goals. This will be the subject of future research. Another direction for future research is that of developing simple algorithms to identify key agents in the opinion network. Finally, a question of mathematical interest is the following:

Given an arbitrary non-ergodic time-varying chain, what is the sparsest time-invariant chain such that sum of the two chains becomes ergodic? There seems to be a relationship between the sparsity index of the corresponding graph of the sparsest time-invariant chain and the $\text{rank}$ of the time-varying chain.


\bibliography{EGCsep14}
\bibliographystyle{IEEEtran}


\end{document}